\documentclass{amsart}\usepackage{amsthm,amssymb,bm,graphicx}

\theoremstyle{plain}
\newtheorem{Thm}{Theorem}
\newtheorem{Prop}[Thm]{Proposition}
\newtheorem{Lem}[Thm]{Lemma}

\newcommand{\z}{\textstyle}
\newcommand{\s}[1]{\scalebox{0.9}{$#1$}}\newcommand{\sq}[1]{\s{\sqrt{#1}}}
\newcommand{\vk}{\varkappa}
\newcommand{\ia}{{\rm a}}
\newcommand{\iC}{{\rm C}}\newcommand{\iE}{{\rm E}}\newcommand{\iF}{{\rm F}}

\newcommand{\cM}{\mathcal{M}}\newcommand{\cP}{\mathcal{P}}\newcommand{\cS}{\mathcal{S}}
\newcommand{\C}{\mathbb{C}}\renewcommand{\P}{\mathbb{P}}
\newcommand{\Q}{\mathbb{Q}}\newcommand{\Z}{\mathbb{Z}}

\newcommand{\BM}{\begin{smallmatrix}}\newcommand{\EM}{\end{smallmatrix}}
\newcommand{\BC}{\begin{cases}}\newcommand{\EC}{\end{cases}}
\newcommand{\VS}[1]{\z\sum\limits_{#1}}
\newcommand{\VP}[1]{\z\prod\limits_{#1}}
\newcommand{\VT}[1]{\z\bigoplus\limits_{#1}}
\newcommand{\6}{\;\;\;\;\;\;}

\begin{document}

\title{An explicit representation of primitive forms}
\author{SAITO Hayato}
\author{SUDA Tomohiko}
\address{Department of Mathematics, Hokkaido University,
Sapporo 060-0810, Japan}
\email{s053019@math.sci.hokudai.ac.jp}

\maketitle

\section{Introduction}

For positive integers $k$ and $N$, let $\cM_k(N)$ be the space of all modular forms of weight $k$ with respect to the congruence subgroup
\[\Gamma_0(N)=\{(\BM a&b\\c&d\EM)\in{\rm SL}_2(\Z) \,|\, c\in N\Z\},\]
and $\cS_k(N)$ be the space of all cuspforms in $\cM_k(N)$. We denote by $\cS_k^0(N)$ the subspace of $\cS_k(N)$ consisting of all newforms (cf. \cite[p162]{Mi}), and $\cP_k(N)$ the set of all primitive forms (cf. \cite[p164]{Mi}) in $\cS_k^0(N)$. The purpose of this paper is to represent $\cP_k(N)$ in terms of some Eisenstein series, for $N=1,2,3,4,6,8,9$.

We put $\P(N)$ the set of all prime divisors of $N$, and $N^\times=\VP{p\in\P(N),\,p^2\nmid N}p$. For each divisor $i$ of $N^\times$, we define
\[\cP_k(N;i)=\big\{f\in\cP_k(N) \,\big|\, p\in\P(i)\Longrightarrow\ia_p(f)=-p^\vk,\;p\in\P(\z\frac{N^\times}i)\Longrightarrow\ia_p(f)=p^\vk\big\},\]
where $\ia_l(f)$ is the $l$-th Fourier coefficient of $f$ and $\vk=\frac k2-1$. Note that for each $f\in\cP_k(N)$ and $p\in\P(N)$,
\[\ia_p(f)=\BC \pm p^\vk &\text{if }p^2\nmid N, \\ 0 &\text{if }p^2|N, \EC\]
(cf. Miyake[6, Theorem 4.6.17]). We can calculate $\ia_l(\sum \cP_k(N;i))$, for each $l\ge1$, by my results \cite{S}. We consider the twist
\[f\otimes\rho=\VS{n}\rho(n)\ia_n(f)q^n\]
for a Dirichlet character $\rho$ mod $N$, and define
\[\cP_k^1(N)=\big\{f\otimes\rho \,\big|\, f\in\z\bigcup\limits_{N\neq M|N}\cP_k(M),\,\bm1_S\neq\rho:(\Z/S\Z)^\times\to\C^\times\big\}\]
where $S=\max\{n \,|\, N\in n^2\Z\}$, and $\cP_k^0(N)=\cP_k(N)\setminus\cP_k^1(N)$. Then, we see $\cP_k^1(N)=\emptyset$ for $N=1,2,3,4,6,8$, and
\[\cP_k^1(9)=\{f\otimes\rho_3 \,|\, f\in\cP_k(1)\cup\cP_k(3)\}\subset\cP_k(9),\]
where $\rho_3$ is the non-trivial Dirichlet character mod 3 (cf. \cite[Thorem 3.1 and Corollaly 3.1]{AL}). We note that each $f\in\cP_k^0(9)$ has CM by $\rho_3$. We remark $\cP_k^0(16)=\emptyset$ for example, see also \cite{FS}. 

The graded ring $\cM(N)=\VT{k}\cM_k(N)$ is studied in \cite{SS}. We use the results of the paper. We ragard $\cM(N)$ as a subring of $\C[[q]]$. Note that $\cS(N)=\VT{k}\cS_k(N)$ is an ideal of $\cM(N)$.

\section{first primitive forms}

For $n\ge0$, we write down the following formulas (cf. \cite[Example 4,20,21]{S}):
\[\begin{array}{|c|c|cc|c|c|}\hline
k&\#\cP_k(1)&\#\cP_k(3;1)&\#\cP_k(3;3)&\#\cP_k(4)&\#\cP_k^0(9)\\\hline
2+12n&-\iota+n&\iota+n&n&n&2n\\
4+12n&n&n&n&n&1+2n\\
6+12n&n&1+n&n&1+n&2n\\\hline
8+12n&n&n&1+n&n&2+2n\\
10+12n&n&1+n&1+n&1+n&1+2n\\
12+12n&1+n&n&1+n&1+n&2+2n\\\hline
\end{array}\]
\[\begin{array}{|c|cc|cccc|}\hline
k&\#\cP_k(2;1)&\#\cP_k(2;2)&\#\cP_k(6;1)&\#\cP_k(6;2)&\#\cP_k(6;3)&\#\cP_k(6;6)\\\hline
2+24n&\iota+n&n&-\iota+n&n&n&n\\
4+24n&n&n&n&n&n&1+n\\
6+24n&n&n&n&n&1+n&n\\
8+24n&n&1+n&1+n&n&n&n\\
10+24n&1+n&n&n&1+n&n&n\\
12+24n&n&n&1+n&1+n&n&1+n\\\hline
14+24n&1+n&1+n&n&n&1+n&n\\
16+24n&n&1+n&1+n&n&1+n&1+n\\
18+24n&1+n&n&n&1+n&1+n&1+n\\
20+24n&1+n&1+n&1+n&1+n&n&1+n\\
22+24n&1+n&1+n&1+n&1+n&1+n&n\\
24+24n&n&1+n&2+n&1+n&1+n&1+n\\\hline
\end{array}\]
\[\#\cP_k(8)=\z\frac14(k-1)+\frac14(-1)^{\frac k2},\]
where $\iota=\min(1,n)$.

Let $B_k$ be the $k$-th Bernoulli number and
\[\iE_k=1-\z\frac{2k}{B_k}\VS{n=1}^\infty\Big(\VS{d|n}d^{k-1}\Big)q^n.\]
Then we see $\iE_k\in\cM_k(1)$ for $k\ge4$ and $\iE_8=\iE_4^2$, $\iE_{10}=\iE_4\iE_6$, $\iE_{14}=\iE_4^2\iE_6$. For $N>1$, we put
\[\iC_N=\z\frac1{(N-1,24)}(N\iE_2^{(N)}-\iE_2)\in\cM_2(N).\]
where $f^{(h)}(q)=f(q^h)$, and
\begin{align*}\iF_3&=1+6\VS{n}\Big(\VS{d|n}\rho_3(d)\Big)q^n\in\cM_1(3,\rho_3),\\
\iF_4&=1+4\VS{n}\Big(\VS{d|n}\rho_4(d)\Big)q^n\in\cM_1(4,\rho_4).\end{align*}
Moreover, put
\begin{align*}\alpha_6&=\z\frac16(\iF_3-\iF_3^{(2)}),\\
\alpha_8&=\z\frac14(\iF_4-\iF_4^{(2)}),\\
\alpha_9&=\z\frac16(\iF_3-\iF_3^{(3)}),\end{align*}
and $\alpha_4=\iF_4^{(2)}\alpha_8$, $\alpha_2=\iC_4\alpha_4$, $\alpha_3=\iC_9\alpha_9$.

For a graded ring $R=\VT{k=0}^\infty R_k$ and $n_1,\cdots,n_r>0$, we define
\[S=R[X_1,\cdots,X_m]^{[n_1,\cdots,n_m]}\]
to be a ring $R[X_1,\cdots,X_m]$ graded as $X_i\in S_{n_i}$. Also for a graded-ring $R$, we define a graded-ring
\[R_{2\Z}=\VT{k:{\rm even}}R_k.\]
\begin{Prop}We have naturally
\begin{align*}\cM(1)&\simeq\C[\iE_4,\iE_6]^{[4,6]},\\
\cM(2)&\simeq\C[\iC_2,\alpha_2]^{[2,4]},\\
\cM(3)&\simeq(\C[\iF_3,\alpha_3]^{[1,3]})_{2\Z},\\
\cM(4)&\simeq(\C[\iF_4,\alpha_4]^{[1,2]})_{2\Z},\\
\cM(6)&\simeq(\C[\iF_3,\alpha_6]^{[1,1]})_{2\Z},\\
\cM(8)&\simeq(\C[\iF_4,\alpha_8]^{[1,1]})_{2\Z},\\
\cM(9)&\simeq(\C[\iF_3,\alpha_9]^{[1,1]})_{2\Z}.\end{align*}
\end{Prop}

We put
\begin{align*}\Delta_1&=\z\frac1{12^3}(\iE_4^3-\iE_6^2),\\
\Delta_2&=\alpha_2(\iC_2^2-64\alpha_2),\\
\Delta_3&=\alpha_3(\iF_3^3-27\alpha_3),\\
\Delta_4&=\alpha_2(\iC_4-16\alpha_4),\\
\Delta_6&=\alpha_6(\iF_3-3\alpha_6)(\iF_3-4\alpha_6)(\iF_3-12\alpha_6),\\
\Delta_8&=\alpha_4\iF_4(\iF_4-8\alpha_8),\\
\Delta_9&=\alpha_3(\iF_3-9\alpha_9).\end{align*}

\begin{Thm}We get
\begin{align*}\cP_{12}(1)&=\{\Delta_1\},\\
\cP_8(2;2)&=\{\Delta_2\},\\
\cP_6(3;1)&=\{\Delta_3\},\\
\cP_6(4)&=\{\Delta_4\},\\
\cP_4(6;6)&=\{\Delta_6\},\\
\cP_4(8)&=\{\Delta_8\},\\
\cP_4^0(9)&=\{\Delta_9\}.\end{align*}
\end{Thm}
\begin{proof}For example, the first assertion follows from
\[\z\sum\cP_{12}(1)-\Delta_1\in\cM_{12}(1)\cap\C[[q]]q^2=\{0\}.\]
\end{proof}

\begin{Thm}We get
\begin{align*}\cS(1)&=(\Delta_1),\\
\cS(2)&=(\Delta_2),\\
\cS(3)&=(\Delta_3),\\
\cS(4)&=(\Delta_4),\\
\cS(6)&=(\Delta_6),\\
\cS(8)&=(\Delta_8),\\
\cS(9)&=(\Delta_9).\end{align*}
\end{Thm}
\begin{proof}
We see $\cS_{12+k}(1)=\Delta_1\cM_k(1)$ since $\dim\cS_{12+k}(1)=\dim\cM_k(1)$, and we get the first assertion, for example.
\end{proof}

\section{Further primitive forms}

At first, for $f\in\cP_k(N)$ and a prime number $p\nmid N$, note
\[\ia_p(f)\ia_{p^{n-1}}(f)=\ia_{p^n}(f)+p^{k-1}\ia_{p^{n-2}}(f)\]
(cf. \cite[Lemma 4.5.7(2) and (4.5.27)]{Mi}), thus
\[\ia_p(f)^2=\ia_{p^2}(f)+p^{k-1},\]
\[\ia_p(f)^3=\ia_{p^3}(f)+2p^{k-1}\ia_p(f),\]
\[\ia_p(f)^4=\ia_{p^4}(f)+3p^{k-1}\ia_{p^2}(f)+2p^{2(k-1)},\]
\[\ia_p(f)^5=\ia_{p^5}(f)+4p^{k-1}\ia_{p^3}(f)+5p^{2(k-1)}\ia_p(f).\]

\begin{Lem}If $f,g\in\cP_k(N)$ and $\sigma=\frac{f+g}2$, $\phi=\frac{f-g}2$, then for a prime number $p\nmid N$ we see
\[\ia_p(\phi)^2=\z-\ia_p(\sigma)^2+\ia_{p^2}(\sigma)+p^{k-1},\]
\[\ia_{p^2}(\phi)=\z2\ia_p(\sigma)\ia_p(\phi),\]
\[\ia_{p^3}(\phi)=\z\big(2\ia_p(\sigma)^2+\ia_{p^2}(\sigma)-p^{k-1}\big)\ia_p(\phi).\]
Moreover, if $(l,m)=1$, then we see
\[\ia_l(\phi)\ia_m(\phi)=\ia_{lm}(\sigma)-\ia_l(\sigma)\ia_m(\sigma),\]
\[\ia_{lm}(\phi)=\ia_l(\sigma)\ia_m(\phi)+\ia_l(\phi)\ia_m(\sigma).\]
\end{Lem}
\begin{proof}
The first two assertions follow from
\begin{align*}\ia_p(\sigma+\phi)^2&=\ia_{p^2}(\sigma+\phi)^2+p^{k-1},\\
\ia_p(\sigma-\phi)^2&=\ia_{p^2}(\sigma-\phi)^2+p^{k-1}.\end{align*}
The last two assertions follow from
\begin{align*}\ia_l(\sigma+\phi)\ia_m(\sigma+\phi)&=\ia_{lm}(\sigma+\phi),\\
\ia_l(\sigma-\phi)\ia_m(\sigma-\phi)&=\ia_{lm}(\sigma-\phi).\end{align*}
\end{proof}
\begin{Prop}With $P_k=\cP_k(1)/\Delta_1$, $d=12\Delta_1$, we get
\begin{align*}P_{12}&=\{1\},\\
P_{16}&=\{\iE_4\},\\
P_{18}&=\{\iE_6\},\\
P_{20}&=\{\iE_8\},\\
P_{22}&=\{\iE_{10}\},\\
P_{26}&=\{\iE_{14}\},\end{align*}
\begin{align*}P_{24}&=\big\{\iE_4^3-(13\pm\sq{144169})d\big\},\\
P_{28}&=\big\{\iE_4(\iE_4^3-9(47\pm\sq{18209})d)\big\},\\
P_{30}&=\big\{\iE_6(\iE_4^3+8(43\pm\sq{51349})d)\big\},\\
P_{32}&=\big\{\iE_8(\iE_4^3+(1567\pm\sq{18295489})d)\big\},\\
P_{34}&=\big\{\iE_{10}(\iE_4^3-6(851\pm\sq{2356201})d)\big\},\\
P_{38}&=\big\{\iE_{14}(\iE_4^3-4(2039\pm\sq{63737521})d)\big\}.\end{align*}
\end{Prop}
\begin{proof}We show the assertion for $P_{24}$: Note $\cS_{24}(1)=\Delta_1(\C\iE_4^3\oplus\C\Delta_1)$. We see
\begin{align*}\big\{\z\frac{f+g}2\,\big|\,\cP_{24}(1)=\{f,g\}\big\}&\subset\cS_{24}(1)\cap(q+540q^2+\C[[q]]q^3)\\
&=\{\Delta_1(\iE_4^3-13d)\},\end{align*}
\begin{align*}\big\{\z\frac{f-g}2\,\big|\,\cP_{24}(1)=\{f,g\}\big\}&\subset\{\phi\in\cS_{24}(1) \,|\, \ia_1(\phi)=0,\,\ia_2(\phi)^2=12^2\cdot144169\}\\
&=\{\pm\sq{144169}\Delta_1d\}.\end{align*}
\end{proof}
\smallskip

\begin{Prop}With $P_{k;i}=\cP_k(2;i)/\Delta_2$, $d=8\Delta_2$, $G=\iC_2^2-128\alpha_2$ and
\[H=\iC_2^4+18d=\iE_4\iE_4^{(2)},\]
\[I=\iC_2^4-81d=\iE_6\iE_6^{(2)}/\iC_2^2,\]
we get
\[P_{8;2}=\{1\},\]
\[P_{10;1}=\{\iC_2\},\]
\begin{align*}P_{14;1}&=\{\iC_2^3\},\\
\6P_{14;2}&=\{\iC_2G\},\end{align*}
\[P_{16;2}=\{\iC_2^4-27d\},\]
\[P_{18;1}=\{\iC_2H\},\]
\begin{align*}P_{20;1}&=\{G(\iC_2^4+63d)\},\\
P_{20;2}&=\{\iC_2^2I\},\end{align*}
\begin{align*}P_{22;1}&=\{\iC_2^3(\iC_2^4+108d)\},\\
P_{22;2}&=\{\iC_2G(\iC_2^4-132d)\},\end{align*}
\[P_{24;2}=\{H(\iC_2^4-297d)\},\]
\begin{align*}
P_{26;1}&=\big\{\iC_2(H(\iC_2^4+468d)-75(161\pm\sq{106705})d^2)\big\},\\
P_{26;2}&=\{\iC_2^3G(\iC_2^4-522d)\},\end{align*}
\begin{align*}P_{28;1}&=\{G(\iC_2^8+1011\iC_2^4d-126^2d^2)\},\\
P_{28;2}&=\{\iC_2^2I(\iC_2^4-972d)\},\end{align*}
\begin{align*}P_{30;1}&=\{\iC_2^3H(\iC_2^4+1998d)\},\\
P_{30;2}&=\{\iC_2GH(\iC_2^4-2082d)\},\end{align*}
\begin{align*}P_{32;1}&=\{\iC_2^2GI(\iC_2^4+4158d)\},\\
P_{32;2}&=\big\{\iC_2^4I(\iC_2^4-4050d)-3(35711\pm5\sq{987507049})(\iC_2^4-9d)d^2\big\},\end{align*}
\begin{align*}
P_{34;1}&=\big\{\iC_2(H^2(\iC_2^4+8118d)-165(2617\pm\sq{79829689})(\iC_2^4-12d)d^2)\big\},\\
P_{34;2}&=\{\iC_2^3G(\iC_2^8-8214\iC_2^4d-2004426d^2)\},\end{align*}
\begin{align*}P_{36;1}&=\{GH(\iC_2^8+16341\iC_2^4d+2008314d^2)\},\\
P_{36;2}&=\{\iC_2^2HI(\iC_2^4-16362d)\},\end{align*}
\begin{align*}P_{38;1}&=\big\{\iC_2^3(I^2(\iC_2^4+32886d)+3(128731\pm10\sq{223572801841})(\iC_2^4-18d)d^2)\big\},\\
P_{38;2}&=\big\{\iC_2G(H(\iC_2^8-32814\iC_2^4d+6800274d^2)-270(8936\pm\sq{3026574721})(\iC_2^4-2d)d^2)\big\},\end{align*}
\begin{align*}P_{40;1}&=\{\iC_2^2GI(\iC_2^8+65586\iC_2^4d-7358526d^2)\},\\
P_{40;2}&=\big\{\iC_2^4I(\iC_2^8-65502\iC_2^4d-7667622d^2)\\[-2pt]
&\6+3(2278733\pm5\sq{4202094647521})(\iC_2^8-21\iC_2^4d+324d^2)d^2\big\},\end{align*}
\begin{align*}P_{42;1}&=\big\{\iC_2H(\iC_2^{12}+131004\iC_2^8d + 62037837\iC_2^4d^2 - 597293298d^3\\[-2pt]
&\6\pm45\sq{4559670239569}(\iC_2^4-42d)d^2)\big\},\\
P_{42;2}&=\{\iC_2^3GH(\iC_2^8-131124\iC_2^4d+83917944d^2)\},\end{align*}
\begin{align*}P_{44;1}&=\big\{G(\iC_2^4I(\iC_2^8+262188\iC_2^4d-125370504d^2)\\[-2pt]
&\6-54(633841\pm5\sq{1589985537001})(\iC_2^8-11\iC_2^4d-294d^2)d^2)\big\},\\
P_{44;2}&=\big\{\iC_2^2I(H(\iC_2^8-262134\iC_2^4d-221008986d^2)\\[-2pt]
&\6+330(344132\pm\sq{97578078049})(\iC_2^4+54d)d^2)\big\},\end{align*}
\begin{align*}P_{46;1}&=\big\{\iC_2^3(H^2(\iC_2^8+524196\iC_2^4d+502603704d^2)\\[-2pt]
&\6-30(2545093\pm\sq{800679089088649})(\iC_2^8-30\iC_2^4d+2106d^2)d^2)\big\},\\
P_{46;2}&=\big\{\iC_2G(H^2(\iC_2^8-524364\iC_2^4d+698541624d^2)\\[-2pt]
&\6-44550(27173\pm\sq{655098313})(\iC_2^8-14\iC_2^4d+674d^2)d^2)\big\},\end{align*}
\begin{align*}P_{48;1}&=\{\iC_2^2GHI(\iC_2^8+1048596\iC_2^4d-3037809096d^2)\},\\[-2pt]
P_{48;2}&=\big\{H(\iC_2^4I(\iC_2^8-1048572\iC_2d-4492325448d^2)\\[-2pt]
&\6+27(40221421\pm5\sq{23589383914321})(5\iC_2^8-255\iC_2^4d-1782d^2)d^2)\big\},\end{align*}
\begin{align*}P_{50;2}&=\big\{\iC_2^3G(\iC_2^{16}-2097198\iC_2^{12}d+2264918139\iC_2^8d^2+505350323388\iC_2^4d^3\\[-2pt]
&\6+6398757749736d^4\pm675\sq{88518163840129})(\iC_2^8-20\iC_2^4d+5496d^2)d^2)\big\},\end{align*}
\begin{align*}P_{52;1}&=\big\{G(\iC_2^{20}+4194255\iC_2^{16}d+6799160250\iC_2^{12}d^2+2040440049090\iC_2^8d^3\\[-2pt]
&\6-83194886921400\iC_2^4d^4+293562746136828d^5\\[-2pt]
&\6\pm2430\sq{1252614454081}(\iC_2^{12}-23\iC_2^8d-75580\iC_2^4d^2-161406d^3)d^2)\big\},\\
P_{52;2}&=\big\{\iC_2^2I(H^2(\iC_2^8-4194324\iC_2^4d-364357891656d^2)\\[-2pt]
&\6+2550(8439457\pm\sq{506427715249})(17\iC_2^8+714\iC_2^4d+4374d^2)d^2)\big\},\end{align*}
\begin{align*}P_{54;1}&=\big\{\iC_2^3H(I^2(\iC_2^8+8388684\iC_2^4d-86315938632d^2)\\[-2pt]
&\6+99(870246887\pm10\sq{6512129075798161})(\iC_2^8-60\iC_2^4d+4536d^2)d^2)\big\},\\
P_{54;2}&=\big\{\iC_2GH(\iC_2^{16}-8388678\iC_2^{12}d-10801650186\iC_2^8d^2+9280605732048\iC_2^4d^3\\[-2pt]
&\6+87094806847056d^4\pm2430\sq{61958883786409}(3\iC_2^8-132\iC_2^4d-21848d^2)d^2)\big\},\end{align*}
\begin{align*}P_{56;1}&=\big\{\iC_2^2GI(H(\iC_2^{12}+16777224\iC_2^8d-176195287908\iC_2^4d^2+29236247952912d^3)\\[-2pt]
&\6+1485(177049117\pm\sq{26906326547864689})(\iC_2^8+52\iC_2^4d-3528d^2)d^2)\big\},\end{align*}
\begin{align*}P_{58;2}&=\big\{\iC_2^3G(\iC_2^{20}-33554490\iC_2^{16}d-375498553815\iC_2^{12}d^2+152782877730240\iC_2^8d^3\\[-2pt]
&\6-3297806265873780\iC_2^4d^4+206104239255806928d^5\\[-2pt]
&\6\pm405\sq{3104405074519849}(23\iC_2^{12}-736\iC_2^8d+1009348\iC_2^4d^2-18089136d^3)d^2)\big\},\end{align*}
\begin{align*}P_{60;1}&=\big\{GH(\iC_2^4I(\iC_2^{12}+67108866\iC_2^8d-948879464328\iC_2^4d^2+536559192167544d^3)\\[-2pt]
&\6-162(452557207\pm5\sq{67855854004234321})\\[-2pt]
&\6(9\iC_2^{12}-477\iC_2^8d+452284\iC_2^4d^2-2677752d^3)d^2)\big\},\\
P_{60;2}&=\big\{\iC_2^2HI(\iC_2^{16}-67108878\iC_2^{12}d+953981895834\iC_2^8d^2+574215577640208\iC_2^4d^3\\[-2pt]
&\6-10988934566025744d^4\pm90\sq{163243373863610946481}(\iC_2^8+12\iC_2^4d+25272d^2)d^2)\big\},\end{align*}
\begin{align*}P_{64;1}&=\big\{\iC_2^2GI(H^2(\iC_2^{12}+268435434\iC_2^8d-6298913770848\iC_2^4d^2+8982842896748952d^3)\\[-2pt]
&\6+405(88418993\pm\sq{303998927727988249})\\[-2pt]
&\6(79\iC_2^{12}+3160\iC_2^8d-4632028\iC_2^4d^2+29910384d^3)d^2)\big\},\end{align*}
\begin{align*}P_{66;2}&=\big\{\iC_2^3GH(\iC_2^{20}-536871000\iC_2^{16}d+9017296260795\iC_2^{12}d^2+35470940529320190\iC_2^8d^3\\[-2pt]
&\6-10876886350087762080\iC_2^4d^4+291830643256197988008d^5\\[-2pt]
&\6\pm4455\sq{7845027215820649}(111\iC_2^{12}-6882\iC_2^8d-38096824\iC_2^4d^2+443457288d^3)d^2)\big\},\end{align*}
\begin{align*}P_{72;1}&=\big\{\iC_2^2GHI(\iC_2^{20}+4294967280\iC_2^{16}d-572595326890155\iC_2^{12}d^2\\[-2pt]
&\6+2300129009375739930\iC_2^8d^3+541770126895663860600\iC_2^4d^4\\[-2pt]
&\6-13229635827326598501192d^5\pm18225\sq{253576158163875107521}\\[-2pt]
&\6(3\iC_2^{12}+30\iC_2^8d-5450816\iC_2^4d^2+5366088d^3)d^2)\big\},\end{align*}
\end{Prop}
\begin{proof}The assertion for $P_{26;1}$ follows from
\begin{align*}&\big\{\z\frac{f+g}2\,\big|\,\cP_{26}(2;1)=\{f,g\}\big\}\\
&\subset\cS_{26}(2)\cap(\s{q+2^{12}q^2+189924q^3+4^{12}q^4+370976550q^5+\C[[q]]q^6})\\
&=\{\Delta_2\iC_2(H(\iC_2^4+468d)-75\cdot161d^2)\},\end{align*}
\begin{align*}&\big\{\z\frac{f-g}2\,\big|\,\cP_{26}(2;1)=\{f,g\}\big\}\\
&\subset\{\phi\in\cS_{26}(2) \,|\, \s{\ia_1(\phi)=\ia_2(\phi)=\ia_4(\phi)=0,\,\ia_3(\phi)^2=4800^2\cdot106705,\,\ia_5(\phi)=-324\ia_3(\phi)}\}\\
&=\{\pm75\sq{106705}\Delta_2\iC_2d^2\}.\end{align*}
\end{proof}

\begin{Prop}With  $P_{k;i}=\cP_k(3;i)/\Delta_3$, $d=3\Delta_3$, $G=\iF_3(\iF_3^3-54\alpha_3)$ and
\[I=\iC_3^3+64d=\iE_4\iE_4^{(3)}/\iC_3,\]
we get
\[P_{6;1}=\{1\},\]
\[P_{8;3}=\{\iC_3\},\]
\begin{align*}P_{10;1}&=\{\iC_3^2\},\\
P_{10;3}&=\{G\},\end{align*}
\[P_{12;3}=\{\iC_3^3+16d\},\]
\begin{align*}P_{14;1}&=\big\{\iC_3(\iC_3^3-(23\pm\sq{1969})d)\big\},\\
P_{14;3}&=\{\iC_3^2G\},\end{align*}
\begin{align*}P_{16;1}&=\{G(\iC_3^3-24d)\},\\
P_{16;3}&=\{\iC_3^2(\iC_3^3-96d)\},\end{align*}
\begin{align*}P_{18;1}&=\big\{\iC_3^3I+(13\pm\sq{14569})(\iC_3^3-8d)d\big\},\\
P_{18;3}&=\{\iC_3GI\},\end{align*}
\begin{align*}P_{20;1}&=\{\iC_3^2G(\iC_3^3-376d)\},\\
P_{20;3}&=\big\{\iC_3(\iC_3^6+91\iC_3^3d-3588d^2\pm\sq{87481}(\iC_3^3-12d)d)\big\},\end{align*}
\begin{align*}P_{22;1}&=\big\{\iC_3^2(I^2-(47\pm21\sq{649})(\iC_3^3-16d)d)\big\},\\
P_{22;3}&=\big\{G(\iC_3^6-198\iC_3^3d-6924d^2\pm762(\iC_3^3+2d)d)\big\},\end{align*}
\begin{align*}P_{24;1}&=\{\iC_3GI(\iC_3^3+296d)\},\\
P_{24;3}&=\big\{\iC_3^3I(\iC_3^3-256d)-(49\pm\sq{530401})(\iC_3^6-20\iC_3^3d-192d^2)d\big\},\end{align*}
\[P_{26;3}=\big\{\iC_3^2G(\iC_3^6-74\iC_3^3d-58692d^2\pm2\sq{1287001}(\iC_3^3-6d)d)\big\},\]
\begin{align*}P_{28;1}&=\big\{G(\iC_3^9+504\iC_3^6d+171936\iC_3^3d^2+807168d^3\pm48\sq{6469}(\iC_3^6-10\iC_3^3d+368d^2)d)\big\},\\
P_{28;3}&=\big\{\iC_3^2(\iC_3^9+3555\iC_3^6d-277524\iC_3^3d^2+6354048d^3\pm21\sq{30001}(\iC_3^6-28\iC_3^3d+2432d^2)d)\big\},\end{align*}
\[P_{30;3}=\big\{\iC_3GI(\iC_3^6-7598\iC_3^3d+60036d^2\pm14\sq{77089}(\iC_3^3-78d)d)\big\},\]
\begin{align*}P_{32;1}&=\big\{\iC_3^2G(I(\iC_3^6+160680\iC_3^3d-11582976d^2)\\[-2pt]
&\6-4(41841\pm\sq{11501281})(\iC_3^6-18\iC_3^3d-4368d^2)d)\big\},\end{align*}
\begin{align*}P_{36;1}&=\big\{\iC_3GI(\iC_3^9-10256\iC_3^6d+15269616\iC_3^3d^2-92214912d^3\\[-2pt]
&\6\pm56\sq{2196841}(\iC_3^6-90\iC_3^3d-912d^2)d)\big\}.\end{align*}
\end{Prop}
\begin{proof}The assertion for $P_{14;1}$ follows from
\begin{align*}\big\{\z\frac{f+g}2\,\big|\,\cP_{14}(3;1)=\{f,g\}\big\}&\subset\cS_{14}(3)\cap(\s{q-27q^2+3^6q^3+\C[[q]]q^4})\\
&=\{\Delta_3\iC_3(\iC_3^3-23d)\},\end{align*}
\begin{align*}\big\{\z\frac{f-g}2\,\big|\,\cP_{14}(3;1)=\{f,g\}\big\}&\subset\{\phi\in\cS_{14}(3) \,|\, \s{\ia_1(\phi)=\ia_3(\phi)=0,\,\ia_2(\phi)^2=3^2\cdot1969}\}\\
&=\{\pm\sq{1969}\Delta_3d\}.\end{align*}
\end{proof}

\begin{Prop}With $P_k=\cP_k(4)/\Delta_4$, $d=192\Delta_1$, we get
\begin{align*}P_8&=\{1\}^{(2)},\\
P_{10}&=\{\iE_4\}^{(2)},\\
P_{12}&=\{\iE_6\}^{(2)},\\
P_{14}&=\{\iE_8\}^{(2)},\\
P_{16}&=\{\iE_{10}\}^{(2)},\\
P_{20}&=\{\iE_{14}\}^{(2)},\end{align*}
\begin{align*}P_{18}&=\big\{\iE_4^3-(19\pm\sq{9361})d\big\}^{(2)},\\
P_{22}&=\big\{\iE_4(\iE_4^3+2(83\pm7\sq{2161})d)\big\}^{(2)},\\
P_{24}&=\big\{\iE_6(\iE_4^3+(443\pm\sq{2473249})d)\big\}^{(2)},\\
P_{26}&=\big\{\iE_8(\iE_4^3-9(261\pm\sq{358121})d)\big\}^{(2)},\\
P_{28}&=\big\{\iE_{10}(\iE_4^3-2(631\pm7\sq{1059289})d)\big\}^{(2)},\\
P_{32}&=\big\{\iE_{14}(\iE_4^3-3(27089\pm\sq{2595343921})d)\big\}^{(2)},\end{align*}
where $X^{(2)}=\{x^{(2)} \,|\, x\in X\}$.
\end{Prop}
\begin{proof}The assertion for $P_{18}(4)$ follows from
\begin{align*}&\big\{\z\frac{f+g}2\,\big|\,\cP_{18}(4)=\{f,g\}\big\}\\
&\subset\cS_{18}(4)\cap(\s{q-2940q^3+302022q^5+12675080q^7+\C[[q]]q^8})\\
&=\{\Delta_4(\iE_4^3-192\cdot19\Delta_1)^{(2)}\},\end{align*}
\begin{align*}&\big\{\z\frac{f-g}2\,\big|\,\cP_{18}(4)=\{f,g\}\big\}\\
&\subset\{\phi\in\cS_{18}(4) \,|\, \ia_1(\phi)=\ia_2(\phi)=\ia_4(\phi)=\ia_6(\phi)=0,\,\ia_3(\phi)^2=192^2\cdot9361,\\[-1pt]
&\6\6\6\6\ia_5(\phi)=-36\ia_3(\phi),\,\ia_7(\phi)=594\ia_3(\phi)\}\\
&=\{\pm192\sq{9361}\Delta_4\Delta_1^{(2)}\}.\end{align*}
\end{proof}

\begin{Prop}With $P_{k;i}=\cP_k(6;i)/\Delta_6$, $d=2\Delta_6$, $G_n=\frac1{1+n}(\iC_3+n\iC_3^{(2)})$ and
\[C=\z\frac14(\iC_2+3\iC_2^{(3)})=\iF_3\iF_3^{(2)},\]
\[H=C^2+6d=\iC_2\iC_2^{(3)},\]
\[I=C^2-27d=G_{-3}G_{-\frac43},\]
\[J=C^2-50d=\z\frac16(-\iC_2+7\iC_2^{(3)})\frac12(-7\iC_2+9\iC_2^{(3)}),\]
we get
\[P_{4;6}=\{1\},\]
\[P_{6;3}=\{C\},\]
\[P_{8;1}=\{CG_2\},\]
\[P_{10;2}=\{G_2(C^2-15d)\},\]
\begin{align*}P_{12;1}&=\{CG_2H\},\\
P_{12;2}&=\{CG_{-2}^3\},\\
P_{12;6}&=\{C^2I\},\end{align*}
\[P_{14;3}=\{C(C^4+18C^2d-390d^2)\},\]
\begin{align*}P_{16;1}&=\{C^3G_2(C^2+48d)\},\\
P_{16;3}&=\{G_2G_{-2}(C^4+57C^2d-540d^2)\},\\
P_{16;6}&=\{H(C^4-87C^2d+660d^2)\},\end{align*}
\begin{align*}P_{18;2}&=\{C^2G_2I(C^2-120d)\},\\
P_{18;3}&=\{C(C^6+108C^4d-3222C^2d^2+19332d^3)\},\\
P_{18;6}&=\{CG_2G_{-2}(C^4-138C^2d-1242d^2)\},\end{align*}
\begin{align*}P_{20;1}&=\{CG_2H(C^4+228C^2d-180d^2)\},\\
P_{20;2}&=\{CG_{-2}(C^6-270C^4d+6756C^2d^2+29160d^3)\},\\
P_{20;6}&=\{C^8-279C^6d-696C^4d^2+67500C^2d^3-243000d^4\},\end{align*}
\begin{align*}P_{22;1}&=\{G_{-2}(C^8+495C^6d+9714C^4d^2+97704C^2d^3-605880d^4)\},\\
P_{22;2}&=\{G_2(C^8-537C^6d+24138C^4d^2-58968C^2d^3+1104840d^4)\},\\
P_{22;3}&=\{C^3(C^6+486C^4d-24306C^2d^2+344160d^3)\},\end{align*}
\begin{align*}P_{24;1}&=\big\{CG_2(C^8+996C^6d+23004C^4d^2-324(7067\pm\sq{1296640489})d^4)\big\},\\
P_{24;2}&=\{CG_{-2}H(C^6-1050C^4d+64008C^2d^2-164592d^3)\},\\
P_{24;3}&=\{C^2G_2G_{-2}I(C^4+1032C^2d-28620d^2)\},\\
P_{24;6}&=\{C^2HI(C^4-1032C^2d-43020d^2)\},\end{align*}
\begin{align*}P_{26;2}&=\{G_2(C^{10}-2079C^8d+182304C^6d^2-456300C^4d^3+22275000C^2d^4-200475000d^5)\},\\
P_{26;3}&=\{C(C^{10}+2016C^8d-183738C^6d^2+4829436C^4d^3-81831384C^2d^4+153465840d^5)\},\\
P_{26;6}&=\{CG_2G_{-2}(C^8-2070C^6d-102006C^4d^2+3059424C^2d^3-22028760d^4)\},\end{align*}
\begin{align*}P_{28;1}&=\{C^3G_2H(C^6+4056C^4d+264024C^2d^2-341280d^3)\},\\
P_{28;2}&=\{C^3G_{-2}(C^8-4122C^6d+476616C^4d^2+3203568C^2d^3+96759360d^4)\},\\
P_{28;3}&=\{G_2G_{-2}H(C^8+4065C^6d-496512C^4d^2+16081308C^2d^3+41805720d^4)\},\\
P_{28;6}&=\big\{C^2I(C^8-4104C^6d-394296C^4d^2+5175000C^2d^3-27945000d^4)\\[-2pt]
&\6+324(9539\pm\sq{13546577521})(7C^4-49C^2d+60d^2)d^4\big\},\end{align*}
\begin{align*}P_{30;1}&=\{C^2G_{-2}I(C^8+8190C^6d+1236924C^4d^2+50173344C^2d^3-161954640d^4)\},\\
P_{30;2}&=\{C^2G_2I(C^8-8202C^6d+1220508C^4d^2+36796032C^2d^3-225899280d^4)\},\\
P_{30;3}&=\big\{C(G_2^2J(C^8+8206C^6d-1022586C^4d^2+10703840C^2d^3+57342000d^4)\\[-2pt]
&\6-4(90733873\pm81\sq{10891418844289})(C^4-10C^2d+30d^2)d^4)\big\},\\
P_{30;6}&=\{CG_2G_{-2}(C^{10}-8220C^8d-1023462C^6d^2+49302756C^4d^3\\[-2pt]
&\6-162213840C^2d^4+2919849120d^5)\},\end{align*}
\begin{align*}
P_{32;1}&=\big\{CG_2(H(C^{10}+16338C^8d+2949084C^6d^2+28920888C^4d^3-145598472C^2d^4\\[-2pt]
&\6+344068560d^5)+324(369169\pm\sq{227942548026769})(C^4-12C^2d+60d^2)d^4)\big\},\\
P_{32;2}&=\{CG_{-2}H(C^{10}-16422C^8d+4095708C^6d^2+41893272C^4d^3-2940431976C^2d^4\\[-2pt]
&\6+12666384720d^5)\},\\
P_{32;3}&=\{G_2G_{-2}(C^{12}+16353C^{10}d-3980124C^8d^2+197992728C^6d^3-4226320800C^4d^4\\[-2pt]
&\6-3717900000C^2d^5+22307400000d^6)\},\\
P_{32;6}&=\{H(C^{12}-16431C^{10}d-2930988C^8d^2+242453592C^6d^3-3522376800C^4d^4\\[-2pt]
&\6+32343300000C^2d^5-52925400000d^6)\},\end{align*}
\begin{align*}P_{34;1}&=\{G_{-2}(C^{14}+32733C^{12}d+9844608C^{10}d^2+319245480C^8d^3-21739219992C^6d^4\\[-2pt]
&\6+220960914120C^4d^5-838619516160C^2d^6+133945526880d^7)\},\\
P_{34;2}&=\big\{G_2(H(C^{12}-32817C^{10}d+12138894C^8d^2+137804868C^6d^3-11111132880C^4d^4\\[-2pt]
&\6+50000097960C^2d^5-188474940720d^6)\\[-2pt]
&\6+324(406955\pm\sq{28748353542841})(7C^6-105C^4d+1560C^2d^2-1980d^3)d^4)\big\},\\
P_{34;3}&=\{C^3(C^{12}+32724C^{10}d-11973390C^8d^2+875018772C^6d^3+9457558704C^4d^4\\[-2pt]
&\6-402776740512C^2d^5+1395244103040d^6)\},\\
P_{34;6}&=\{C^3G_2G_{-2}(C^{10}-32802C^8d-9876546C^6d^2+744598656C^4d^3-31177759248C^2d^4\\[-2pt]
&\6+248604863040d^5)\},\end{align*}
\begin{align*}P_{36;1}&=\big\{CG_2H(C^{12}+65484C^{10}d+29337012C^8d^2+972806112C^6d^3\\[-2pt]
&\6+26965965960C^4d^4-686899802880C^2d^5+2060699408640d^6\\[-2pt]
&\6\pm3240\sq{1088738810239369}(C^4-24C^2d+72d^2)d^4)\big\},\\
P_{36;2}&=\big\{CG_{-2}(G_2^2J(C^{10}-65522C^8d+30909936C^6d^2+2965143488C^4d^3\\[-2pt]
&\6-15435688560C^2d^4-117729828000d^5)\\[-2pt]
&\6+280(16330387\pm81\sq{102536199529})(23C^6-230C^4d+168C^2d^2-432d^3)d^4)\big\},\\
P_{36;3}&=\{C^2G_2G_{-2}HI(C^8+65520C^6d-32874552C^4d^2+2229716808C^2d^3+163616760d^4)\},\\
P_{36;6}&=\big\{C^2I(G_2^2J(C^8-65504C^6d-34839096C^4d^2+522120200C^2d^3-10174425000d^4)\\[-2pt]
&\6-28(303745651\pm81\sq{304948008306841})(C^4+8C^2d-60d^2)d^4)\big\},\end{align*}
\begin{align*}P_{38;1}&=\{G_{-2}(C^{16}+131031C^{14}d+92399298C^{12}d^2+6027991488C^{10}d^3\\[-2pt]
&\6+130642875000C^8d^4+450897471936C^6d^5-68227587057600C^4d^6\\[-2pt]
&\6+152381849400000C^2d^7-304763698800000d^8)\},\\
P_{38;2}&=\{G_2(C^{16}-131121C^{14}d+102361050C^{12}d^2+5117487552C^{10}d^3\\[-2pt]
&\6+237432632184C^8d^4-10600042361472C^6d^5+74142791812800C^4d^6\\[-2pt]
&\6-212495481000000C^2d^7+764983731600000d^8)\},\\
P_{38;3}&=\big\{C(C^{16}+131022C^{14}d-102490266C^{12}d^2+12164659488C^{10}d^3\\[-2pt]
&\6-380291429196C^8d^4+1177892029224C^6d^5+119740488933960C^4d^6\\[-2pt]
&\6-666864146661120C^2d^7+785994621662880d^8\pm324\sq{251204923691569}\\[-2pt]
&\6(179C^8-3938C^6d+25830C^4d^2-77760C^2d^3-50760d^4)d^4)\big\},\\
P_{38;6}&=\{CG_2G_{-2}(C^{14}-131112C^{12}d-92529162C^{10}d^2+11059273980C^8d^3\\[-2pt]
&\6+297096151968C^6d^4-1692160612560C^4d^5-10440947116080C^2d^6\\[-2pt]
&\6-91376673078240d^7)\},\end{align*}
\begin{align*}
P_{40;1}&=\big\{C^3G_2(H(C^{12}+262086C^{10}d+277197408C^8d^2+21849570720C^6d^3\\[-2pt]
&\6+132122012112C^4d^4-738516509856C^2d^5+1639679708160d^6)\\[-2pt]
&\6-324(6958215291\pm13\sq{692769757131108721})\\[-2pt]
&\6(C^2-12d)(C^4-12C^2d+120d^2)d^4)\big\},\\
P_{40;2}&=\{C^3G_{-2}H(C^{12}-262194C^{10}d+301838592C^8d^2+23702240736C^6d^3\\[-2pt]
&\6-3698614891440C^4d^4+54030754022112C^2d^5-261744045511680d^6)\},\\
P_{40;3}&=\big\{G_2G_{-2}(C^2I(C^{12}+262128C^{10}d-292924404C^8d^2+34997362848C^6d^3\\[-2pt]
&\6+1258474844880C^4d^4+2837559600000C^2d^5-38307054600000d^6)\\[-2pt]
&\6+3240(3673159\pm\sq{528575249909281})\\[-2pt]
&\6(47C^8-705C^6d+312C^4d^2+252C^2d^3-55080d^4)d^4)\big\},\\
P_{40;6}&=\big\{H(C^{16}-262203C^{14}d-276932436C^{12}d^2+47147194908C^{10}d^3\\[-2pt]
&\6-1320886033704C^8d^4+33825221370360C^6d^5-456521575910400C^4d^6\\[-2pt]
&\6+1778443757117280C^2d^7-2609195237668800d^8\pm42120\sq{12925296409}\\[-2pt]
&\6(623C^8-19313C^6d+132360C^4d^2-59652C^2d^3-487080d^4)d^4)\big\},\end{align*}
\begin{align*}P_{42;1}&=\{C^2G_{-2}I(C^{14}+524268C^{12}d+864880668C^{10}d^2+132334868880C^8d^3\\[-2pt]
&\6+914696054928C^6d^4+156094524005760C^4d^5-1202185035360000C^2d^6\\[-2pt]
&\6+1243642865984640d^7)\},\\
P_{42;2}&=\big\{C^2G_2I(C^{14}-524316C^{12}d+882706572C^{10}d^2+125079755472C^8d^3\\[-2pt]
&\6+7181374424148C^6d^4-86498474400000C^4d^5+99353839559040C^2d^6\\[-2pt]
&\6-3652218972538560d^7\pm2268\sq{38351621994905041}\\[-2pt]
&\6(31C^6-4320C^2d^2+6480d^3)d^4)\big\},\\
P_{42;3}&=\big\{C(G_2^2J(C^{14}+524284C^{12}d-870122310C^{10}d^2+129229013732C^8d^3\\[-2pt]
&\6-404642618736C^6d^4-36238382888832C^4d^5-5192088841440C^2d^6\\[-2pt]
&\6-259604442072000d^7)+40(10745740153\pm81\sq{164084246636718769})\\[-2pt]
&\6(5C^{10}-140C^8d+4218C^6d^2-46332C^4d^3+94608C^2d^4-189216d^5)d^4)\big\}\\
&=\big\{C(H^2(C^{14} + 524220C^{12}d - 903675654C^{10}d^2 + 185353038180C^8d^3\\[-2pt]
&\6 - 10204073665776C^6d^4 + 135924676359552C^4d^5 - 100569818907360C^2d^6\\[-2pt]
&\6 + 120683782688832d^7)+324(1820959258\pm10\sq{164084246636718769})\\[-2pt]
&\6(5C^{10}-140C^8d+4218C^6d^2-46332C^4d^3+94608C^2d^4-189216d^5)d^4)\big\}\\
P_{42;6}&=\big\{CG_2G_{-2}(H(C^{14}-524340C^{12}d-848102022C^{10}d^2+170131902660C^8d^3\\[-2pt]
&\6-2424032639280C^6d^4+114766797805056C^4d^5-1761920698666464C^2d^6\\[-2pt]
&\6+4530653225142336d^7)-2268(282530698\pm10\sq{24132013874469601})\\[-2pt]
&\6(5C^8-90C^6d+918C^4d^2-2592C^2d^3+11664d^4)d^4)\big\},\end{align*}
\begin{align*}P_{44;1}&=\big\{CG_2H(C^{16}+1048512C^{14}d+2555321184C^{12}d^2+420402270528C^{10}d^3\\[-2pt]
&\6-14726546206344C^8d^4+827817754381920C^6d^5-9874537975375200C^4d^6\\[-2pt]
&\6+54161454237258240C^2d^7-279196762888612800d^8\pm3240\\[-2pt]
&\6\sq{475120385535394312561}(C^8-36C^6d+420C^4d^2-2304C^2d^3+11880d^4)d^4)\big\},\\
P_{44;2}&=\big\{CG_{-2}(G_2^2J(C^{14}-1048574C^{12}d+2605652100C^{10}d^2+590479485896C^8d^3\\[-2pt]
&\6+6522392457504C^6d^4-54904234242240C^4d^5+1378537859088000C^2d^6\\[-2pt]
&\6-5978897258400000d^7)-280(1630617059\pm81\sq{3665860149759721})\\[-2pt]
&\6(47C^{10}-1034C^8d+27252C^6d^2-180792C^4d^3+566136C^2d^4-71280d^5)d^4)\big\},\\
P_{44;3}&=\{G_2G_{-2}H(C^{16}+1048521C^{14}d-2665418688C^{12}d^2+657964797876C^{10}d^3\\[-2pt]
&\6-139310386385400C^8d^4+2792259036461232C^6d^5-16979446706817600C^4d^6\\[-2pt]
&\6+132419827128600000C^2d^7-264839654257200000d^8)\},\\
P_{44;6}&=\big\{C^2I(C^{14}-1048608C^{12}d-2588873376C^{10}d^2+596439951768C^8d^3\\[-2pt]
&\6-6798187526184C^6d^4+795734192520000C^4d^5-10083167786520000C^2d^6\\[-2pt]
&\6-3390396750000000d^7)+324(81274873\pm7\sq{531233135911921})\\[-2pt]
&\6(121C^{12}-3751C^{10}d+628452C^8d^2-10200168C^6d^3+34689600C^4d^4\\[-2pt]
&\6-89100000C^2d^5+145800000d^6)d^4\big\},\end{align*}

\end{Prop}
\begin{proof}The assertion for $P_{24;1}$ follows from
\begin{align*}&\big\{\z\frac{f+g}2\,\big|\,\cP_{24}(6;1)=\{f,g\}\big\}\\
&\subset\cS_{24}(6)\cap(\s{q+2^{11}q^2+3^{11}q^3+4^{11}q^4+12624078q^5+6^{11}q^6+2882231384q^7+8^{11}q^8}\\[-1pt]
&\6\s{+\,9^{11}q^9+2^{11}12624078q^{10}+508560735012q^{11}+12^{11}q^{12}+1377630542438q^{13}}\\[-1pt]
&\6\s{+\,2^{11}2882231384q^{14}+3^{11}12624078q^{15}+16^{11}q^{16}+29152391885874q^{17}+18^{11}q^{18}}\\[-1pt]
&\6\s{-209169475284580q^{19}+4^{11}12624078q^{20}+3^{11}2882231384q^{21}+\C[[q]]q^{22}})\\
&=\{\Delta_6G_2C(C^8+996C^6d+23004C^4d^2-324\cdot7067d^4)\},\end{align*}
\begin{align*}&\big\{\z\frac{f-g}2\,\big|\,\cP_{24}(6;1)=\{f,g\}\big\}\subset\{\phi\in\cS_{24}(6) \,|\\
&\6\s{\ia_1(\phi)=\ia_2(\phi)=\ia_3(\phi)=\ia_4(\phi)=\ia_6(\phi)=\ia_8(\phi)=\ia_9(\phi)=\ia_{12}(\phi)=\ia_{16}(\phi)=\ia_{18}(\phi)=0,}\\[-1pt]
&\6\s{\ia_5(\phi)^2=5184^2\cdot1296640489,\,\ia_{10}(\phi)=2^{11}\ia_5(\phi),\,\ia_{15}(\phi)=3^{11}\ia_5(\phi),\,\ia_{20}(\phi)=4^{11}\ia_5(\phi),}\\[-1pt]
&\6\s{\ia_7(\phi)=-25\ia_5(\phi),\,\ia_{14}(\phi)=2^{11}\ia_7(\phi),\,\ia_{21}(\phi)=3^{11}\ia_7(\phi),\,\ia_{11}(\phi)=-6050\ia_5(\phi),}\\[-1pt]
&\6\s{\ia_{13}(\phi)=15550\ia_5(\phi),\,\ia_{17}(\phi)=-801550\ia_5(\phi),\,\ia_{19}(\phi)=1361350\ia_5(\phi)}\}\\
&=\{\pm324\sq{1296640489}\Delta_6G_2Cd^4\}.\end{align*}
\end{proof}

\begin{Lem}Let $p$ be a prime such that $p\nmid N$. If $f_1,f_2,f_3\in\cP_k(N)$ and $\sigma=f_1+f_2+f_3$, $A_{m,n}=m\ia_{p^2}(\sigma)+np^{k-1}$, then we see
\begin{align*}&\VP{i\in\{1,2,3\}}(X-\ia_p(f_i))\\
&=X^3-\ia_p(\sigma)X^2+\z\frac12(\ia_p(\sigma)^2-A_{1,3})X\\
&\6-\z\frac16\big(\ia_p(\sigma)^3-\ia_p(\sigma)A_{3,5}+2\ia_{p^3}(\sigma)\big)\end{align*}
in $\C[X]$. If $f_1,f_2,f_3,f_4\in\cP_k(N)$ and $\sigma=f_1+f_2+f_3+f_4$, then
\begin{align*}&\VP{i\in\{1,2,3,4\}}(X-\ia_p(f_i))\\
&=X^4-\ia_p(\sigma)X^3+\z\frac12(\ia_p(\sigma)^2-A_{1,4})X^2\\
&-\z\frac16\big(\ia_p(\sigma)^3-\ia_p(\sigma)A_{3,8}+2\ia_{p^3}(\sigma)\big)X\\
&+\z\frac1{24}\big(\ia_p(\sigma)^4-2\ia_p(\sigma)^2A_{3,4}+8\ia_p(\sigma)\ia_{p^3}(\sigma)+3\ia_{p^2}(\sigma)A_{1,2}-6\ia_{p^4}(\sigma)\big).\end{align*}
If $f_1,f_2,f_3,f_4,f_5\in\cP_k(N)$ and $\sigma=f_1+f_2+f_3+f_4+f_5$, then
\begin{align*}&\VP{i\in\{1,2,3,4,5\}}(X-\ia_p(f_i))\\
&=X^5-\ia_p(\sigma)X^4+\z\frac12(\ia_p(\sigma)^2-A_{1,5})X^3\\
&\6-\z\frac16\big(\ia_p(\sigma)^3-\ia_p(\sigma)A_{3,11}+2\ia_{p^3}(\sigma)\big)X^2\\
&\6+\z\frac1{24}\big(\ia_p(\sigma)^4-2\ia_p(\sigma)^2A_{3,7}+8\ia_p(\sigma)\ia_{p^3}(\sigma)+3A_{1,4,5}-6\ia_{p^4}(\sigma)\big)X\\
&\6-\z\frac1{120}\big(\ia_p(\sigma)^5-10\ia_p(\sigma)^3A_{1,1}+20\ia_p(\sigma)^2\ia_{p^3}(\sigma)+5\ia_p(\sigma)A_{3,4,-1}\\
&\6\6-30\ia_p(\sigma)\ia_{p^4}(\sigma)-5\ia_{p^3}(\sigma)A_{4,1}+24\ia_{p^5}(\sigma)\big),\end{align*}
where $A_{l,m,n}=l\ia_{p^2}(\sigma)^2+m\ia_{p^2}(\sigma)p^{k-1}+np^{2(k-1)}$.
\end{Lem}
\begin{proof}
For the first assertion, set $a_i=\ia_p(f_i)$ and $s_i=a_1^i+a_2^i+a_3^i$, then we see
\[s_1=\ia_p(\sigma),\6s_2=\ia_{p^2}(\sigma)+3p^{k-1},\6s_3=\ia_{p^3}(\sigma)+2p^{k-1}\ia_p(\sigma),\]
and
\[a_1a_2+a_2a_3+a_3a_1=\z\frac12(s_1^2-s_2),\]
\[a_1a_2a_3=\z\frac16(s_1^3-3s_1s_2+2s_3).\]

For the second assertion, set $s_i=a_1^i+a_2^i+a_3^i+a_4^i$,
then we see
\[s_1=\ia_p(\sigma),\6s_2=\ia_{p^2}(\sigma)+4p^{k-1},\]
\[s_3=\ia_{p^3}(\sigma)+2p^{k-1}\ia_p(\sigma),\6s_4=\ia_{p^4}(\sigma)+3p^{k-1}\ia_{p^2}(\sigma)+8p^{2(k-1)}\]
 and
\[a_1a_2a_3a_4=\z\frac1{24}(s_1^4-6s_1^2s_2+3s_2^2+8s_1s_3-6s_4).\]

For the third assertion, set $s_i=a_1^i+a_2^i+a_3^i+a_4^i+a_5^i$,
then we see
\[s_1=\ia_p(\sigma),\6s_2=\ia_{p^2}(\sigma)+5p^{k-1},\]
\[s_3=\ia_{p^3}(\sigma)+2p^{k-1}\ia_p(\sigma),\6s_4=\ia_{p^4}(\sigma)+3p^{k-1}\ia_{p^2}(\sigma)+10p^{2(k-1)},\]
\[s_5=\ia_{p^5}(\sigma)+4p^{k-1}\ia_{p^3}(\sigma)+5p^{2(k-1)}\ia_p(\sigma),\]
and
\[a_1a_2a_3a_4a_5=\z\frac1{120}(s_1^5-10s_1^3s_2+20s_1^2s_3+15s_1s_2^2-30s_1s_4-20s_2s_3+24s_5).\]
\end{proof}

\begin{Prop}With $P_k=\cP_k(8)/\Delta_8$, $d=2^6\Delta_2$ and $G=\iC_2^2-128\alpha_2$, we get
\[P_4=\{1\}^{(2)},\]
\[P_6=\{\iC_2\}^{(2)},\]
\[P_8=\big\{\iC_2^2,G\big\}^{(2)},\]
\[P_{10}=\big\{\iC_2^3,\iC_2G\big\}^{(2)},\]
\[P_{12}=\big\{\iC_2^2G,\iC_2^4-(1\pm\sq{109})d\big\}^{(2)},\]
\[P_{14}=\z\big\{\iC_2^3G,\iC_2(\iC_2^4+(5\pm\sq{781})d)\big\}^{(2)},\]
\[P_{16}=\z\big\{G(\iC_2^4-6(1\pm8)d),\iC_2^2(\iC_2^4-2(17\pm5\sq{58})d)\big\}^{(2)},\]
\[P_{18}=\z\big\{\iC_2G(\iC_2^4+18(5\pm\sq{114})d),\iC_2^3(\iC_2^4-2(5\pm2\sq{2146})d)\big\}^{(2)},\]
\[P_{20}\supset\z\big\{\iC_2^2G(\iC_2^4-3(73\pm5\sq{1453})d)\big\}^{(2)},\]
\[P_{22}\supset\z\big\{\iC_2^3G(\iC_2^4-3(275\pm\sq{358549})d)\big\}^{(2)}.\]
\end{Prop}
\begin{proof}We show the asserion for $P_{18}$: First, note that $\#P_{18}=4$. We put $K_1=\Q(\sq{114})$, $K_2=\Q(\sq{2146})$ and
\[\cP_i=\{f\in\cP_{18}(8) \,|\, \ia_3(f)\in K_i\}\]
for $i=1,2$. We see
\[\VP{f\in\cP_{18}(8)}(X-\ia_3(f))=(X^2-11592X-117696240)(X^2+952X-140413680),\]
by the above Lemma, and thus $\#\cP_1=\#\cP_2=2$. Next, we see
\begin{align*}\VP{f\in\cP_{18}(8)}(\s{X-\ia_5(f)})&=(\s{X^2+791924X-542778388700})(\s{X^2+53620X-505586145500}),\\
\VP{f\in\cP_{18}(8)}(\s{X-\ia_{15}(f)})&=(\s{X^2+25165411920X+63882975503248488000})\\[-10pt]
&\6\cdot(\s{X^2-16902353840X+70991211246670440000}),\end{align*}
and thus
\[f\in\cP_i \Longrightarrow \ia_5(f)\in K_i.\]
Similar results are held for $\ia_7(f),\ia_{11}(f),\ia_{13}(f)$. We see
\begin{align*}\big\{\z\frac{f+g}2 \,\big|\, \cP_1=\{f,g\}\big\}&\subset\cS_{18}(8)\cap(\s{q+5796q^3-395962q^5-9466296q^7+55743309q^9}\\[-2pt]
&\6\s{-235385620q^{11}-1251761714q^{13}-12582705960q^{15}+\C[[q]]q^{17}})\\
&=\{\Delta_8(\iC_2G(\iC_2^4+18\cdot5d))^{(2)}\},\end{align*}
\begin{align*}\big\{\z\frac{f-g}2 \,\big|\, \cP_1=\{f,g\}\big\}&\subset\{\phi\in\cS_{18}(8) \,|\, \s{\ia_1(\phi)=0,\,\ia_{2n}(\phi)=0\text{ for }n\ge1,\,\ia_3(\phi)^2=1152^2\cdot114}\\[-2pt]
&\6\s{\ia_5(\phi)=-68\ia_3(\phi),\,\ia_7(\phi)=-1582\ia_3(\phi),\,\ia_9(\phi)=11592\ia_3(\phi),}\\[-2pt]
&\6\s{\ia_{11}(\phi)=-10269\ia_3(\phi),\,\ia_{13}(\phi)=49948\ia_3(\phi),\,\ia_{15}(\phi)=-790090\ia_3(\phi)}\}\\
&=\{\pm18\sq{114}\Delta_8(\iC_2Gd)^{(2)}\}.\end{align*}
Similarly, we may decide on $\cP_2$.
\end{proof}

\begin{Prop}With
\[P_{k;0}=(\cP_k^0(9)\cap\C[[q^3]]q)/\Delta_9,\]
\[P_{k;*}=(\cP_k^0(9)\setminus\C[[q^3]]q)/\Delta_9,\]
$d_1=6^3\Delta_1$ and $d_9=3\Delta_9^{(\frac13)}$, we get
\[P_{4;0}=\{1\},\]
\[P_{10;0}=\{\iE_6\}^{(3)},\]
\[P_{16;0}=\{\iE_4^3-155d_1\}^{(3)},\]
\[P_{22;0}=\{\iE_6(\iE_4^3-9710d_1)\}^{(3)},\]
\[P_{28;0}=\{\iE_4^6-621385\iE_4^3d_1-8049950d_1^2\}^{(3)},\]
\[P_{34;0}=\{\iE_6(\iE_4^6-39768220\iE_4^3d_1+67950520d_1^2)\}^{(3)},\]
\[P_{40;0}=\{\iE_4^9-2545165815\iE_4^6d_1+1186765957380\iE_4^3d_1^2-3602855202920d_1^3\}^{(3)},\]
\begin{align*}P_{46;0}&=\{\iE_6(\iE_4^9-162890611530\iE_4^6d_1-337997152434120\iE_4^3d_1^2+250735128034600d_1^3)\}^{(3)},\end{align*}
\begin{align*}P_{52;0}&=\{\iE_4^{12}-10424999137445\iE_4^9d_1-40705082453455050\iE_4^6d_1^2\\[-3pt]
&\6+261266383110843400\iE_4^3d_1^3-23007109783000799000d_1^4\}^{(3)},\end{align*}
\begin{align*}P_{58;0}&=\{\iE_6(\iE_4^{12}-667199944795640\iE_4^9d_1+51521275253851509840\iE_4^6d_1^2\\[-3pt]
&\6-209888117459184763400\iE_4^3d_1^3-15677307231304614533000d_1^4)\}^{(3)},\end{align*}
\[P_{8;*}=\big\{\iE_4\pm2\sq{10}d_9\big\}^{(3)},\]
\[P_{12;*}=\big\{\iE_4(\iE_4\pm2\sq{70}d_9)\big\}^{(3)},\]
\[P_{14;*}=\big\{\iE_6(\iE_4\pm4\sq{55}d_9)\big\}^{(3)},\]
\[P_{16;*}=\big\{\iE_4^3+3200d_9^3\pm6\sq{370}\iE_8d_9\big\}^{(3)},\]
\[P_{18;*}=\big\{\iE_{10}(\iE_4\pm4\sq{910}d_9)\big\}^{(3)},\]
\begin{align*}P_{20;*}&=\big\{\iE_4(\iE_4^3+20(367+v)d_9^3)\pm\sq{20095+15v}(2\iE_4^3+(913-v)d_9^3)d_9 \,\big|\, v^2=9\cdot158041\big\}^{(3)},\end{align*}
\[P_{22;*}=\big\{\iE_6(\iE_4^3+70400d_9^3\pm12\sq{3085}\iE_8d_9)\big\}^{(3)},\]
\begin{align*}P_{24;*}&=\big\{\iE_4\big(\iE_4(\iE_4^3+100(1769+v)d_9^3)\\[-3pt]
&\6\pm5\sq{14629+3v}(2\iE_4^3-(1463-v)d_9^3)d_9\big) \,\big|\, v^2=9\cdot1439881\big\}^{(3)},\end{align*}
\begin{align*}P_{26;*}&=\big\{\iE_6\big(\iE_4(\iE_4^3+160(7955+v)d_9^3)\\[-3pt]
&\6\z\pm4\sq{471670+30v}(\iE_4^3-\frac5{17}(235+v)d_9^3)d_9\big) \,\big|\, v^2=9\cdot13716145\big\}^{(3)},\end{align*}
\begin{align*}P_{28;*}&=\big\{\iE_4^6+20(135563+33v)\iE_4^3d_9^3+25600(6367-3v)d_9^6\\[-3pt]
&\6\pm3\sq{640195+55v}\iE_8(2\iE_4^3+(-10367+3v)d_9^4)d_9 \,\big|\, v^2=9\cdot5173169\big\}^{(3)},\end{align*}
\begin{align*}P_{30;*}&=\big\{\iE_{10}\big(\iE_4(\iE_4^3+40(27491+v)d_9^3)\\[-3pt]
&\6\z\pm\sq{15737830+30v}(2\iE_4^3+\frac5{19}(469469-v)d_9^3)d_9\big) \,\big|\, v^2=9\cdot27539759809\big\}^{(3)},\end{align*}
\begin{align*}P_{34;*}&=\big\{\iE_6\big((\iE_4^6+80(3945461+3v)\iE_4^3d_9^3-1408000(32939-3v)d_9^6)\\[-3pt]
&\6\pm3\sq{13203805+5v}\iE_4^2(4\iE_4^3+5(31339-3v)d_9^3)d_9\big)\,\big|\, v^2=9\cdot159012841\big\}^{(3)}.\end{align*}
\end{Prop}

\begin{proof}We show the asserion for $\cP=\cP_{20}^0(9)$: We see $\#\cP=4$ and
\begin{align*}\VP{f\in\cP}(X-\ia_2(f))&=X^4-1446840X^2+108573696000,\\
\VP{f\in\cP}(X-\ia_5(f))&=X^4-80271385943040X^2+868903477448377525862400000,\\
\VP{f\in\cP}(X-\ia_7(f))&=(X^2-83136040X-16216397509785200)^2,\\
\VP{f\in\cP}(X-\ia_{10}(f))&=(X^2+1886538240X-9712881241620480000)^2,\end{align*}
etc. Put
\[\cP_v=\{f\in\cP_{20}(9) \,|\, \ia_2(f)^2=6^2(20095+15v)\},\]
with $v^2=9\cdot158041$, then we see $\#\cP_v=2$. For $f\in\cP_v$, we see
\begin{align*}\ia_4(f)&=\ia_2(f)^2-2^{19}=199132+1620v,\\
\ia_8(f)&=\ia_2(f)^3-2^{20}\ia_2(f)=(-325156+1620v)\ia_2(f),\\
\ia_{16}(f)&=\ia_2(f)^4-3\cdot2^{19}\ia_4(f)-2^{39}=75135928864-204158880v,\end{align*}
\begin{align*}\ia_5(f)^2&=40135692971520-68518586880v,\\
\ia_{10}(f)&=-943269120-8190720v,\end{align*}
\[\ia_7(f)=41568020+336960v,\]
etc. Some calculations complete the proof.
\end{proof}

\end{document}